\documentclass[a4paper,11pt,twoside]{amsart}

\usepackage[left=2.7cm,right=2.7cm,top=3.5cm,bottom=3cm]{geometry}

\newfont{\cyr}{wncyr10 scaled 1100}

\usepackage{amsmath,amssymb,amsthm,amscd,amsfonts}
\usepackage{mathrsfs}
\usepackage{latexsym}
\usepackage{graphicx}
\usepackage[all]{xy}

\newtheorem{thm}{Theorem}[section]
\newtheorem{pro}[thm]{Proposition}
\newtheorem{cor}[thm]{Corollary}
\newtheorem{lem}[thm]{Lemma}

\theoremstyle{remark}
\newtheorem{rem}{Remark}[section]

\theoremstyle{remark}
\newtheorem{rems}[rem]{Remarks}

\theoremstyle{remark}
\newtheorem*{rem-no-num}{Remark}

\theoremstyle{remark}
\newtheorem*{notat}{Notation}

\theoremstyle{remark}

\theoremstyle{remark}

\theoremstyle{definition}
\newtheorem{dfn}{Definition}[section]

\theoremstyle{definition}

\newcommand{\Q}{\mbox{$\mathbb Q$}}

\newcommand{\C}{\mbox{$\mathbb C$}}

\newcommand{\Z}{\mbox{$\mathbb Z$}}

\newcommand{\F}{\mbox{$\mathbb F$}}
\newcommand{\n}{\mbox{$\mathfrak n$}}
\newcommand{\m}{\mbox{$\mathfrak m$}}
\newcommand{\E}{\mbox{$\mathscr E$}}
\newcommand{\cO}{\mbox{$\mathcal O$}}
\newcommand{\cC}{\mbox{$\mathcal C$}}
\newcommand{\gal}[2]{\mbox{$\mathrm{Gal}(#1/#2)$}}
\newcommand{\divexact}{\mbox{$|\!|$}}

\newcommand{\sha}[1]{\mbox{{\cyr{X}}$(#1)$}}
\newcommand{\sel}[2]{\mbox{${\mathrm{Sel}}_{#1}(#2)$}}
\newcommand{\dsel}[2]{\mbox{${{\mathrm{Sel}}_{#1}(#2)}^\ast$}}
\newcommand{\ddsel}[2]{\mbox{${{\mathrm{Sel}}_{#1}(#2)}^{\ast\ast}$}}
\newcommand{\longmono}{\mbox{$\lhook\joinrel\longrightarrow$}}

\newcommand{\longepi}{\mbox{$\relbar\joinrel\twoheadrightarrow$}}
\newcommand{\frob}{\mbox{${\mathrm{Frob}}$}}
\newcommand{\tr}[2]{\mbox{${\mathrm{Tr}}_{#1/#2}$}}

\begin{document}

\title[On ring class eigenspaces of Mordell-Weil groups]{On ring class eigenspaces of Mordell-Weil groups of elliptic curves over global function fields}
\author{Stefano Vigni}
\address{Dipartimento di Matematica, Universit\`a di Milano, Via C. Saldini 50, 20133 Milano, Italy} 
\email{stevigni@mat.unimi.it}

\begin{abstract}
If $E$ is a non-isotrivial elliptic curve over a global function field $F$ of odd characteristic we show that certain Mordell-Weil groups of $E$ have $1$-dimensional $\chi$-eigenspace (with $\chi$ a complex ring class character) provided that the projection onto this eigenspace of a suitable Drinfeld-Heegner point is nonzero. This represents the analogue in the function field setting of a theorem for elliptic curves over $\Q$ due to Bertolini and Darmon, and at the same time is a generalization of the main result proved by Brown in his monograph on Heegner modules. As in the number field case, our proof employs Kolyvagin-type arguments, and the cohomological machinery is started up by the control on the Galois structure of the torsion of $E$ provided by classical results of Igusa in positive characteristic. 
\end{abstract}

\maketitle

\section{Introduction}

Let $\cC_{/\mathbb F_q}$ be a geometrically connected, smooth, projective algebraic curve over a finite field of characteristic $p>2$, and denote $F:=\F_q(\cC)$ and $\cO_{\mathcal C}$ the function field and the structure sheaf of $\cC$, respectively. Let $E_{/F}$ be a non-isotrivial elliptic curve (i.e., $j(E)\notin\bar{\F}_q$) defined over $F$. At the cost of replacing $F$ by a finite, separable extension $F'$, we can assume that there is a closed point $\infty$ of $\cC$ at which $E$ has split multiplicative reduction. Fix such a point, put $d_\infty:=\text{deg}(\infty)$ and denote $A:=\cO_{\mathcal C}(\cC-\{\infty\})=H^0(\cC-\{\infty\},\cO_{\mathcal C})$ the Dedekind domain of the elements of $F$ that are regular outside $\infty$. Then the conductor $\m$ of $E$ can be written as $\m=\n\infty$ with $\n$ an ideal of $A$; we can view $\m$ as an effective divisor on $\cC$, and compute it via Tate's algorithm. Let $F_\infty$ be the completion of $F$ at $\infty$, and let $\mathbf C_\infty$ be the completion of an algebraic closure $\bar{F}_\infty$ of $F_\infty$.

Choose an imaginary quadratic extension $K$ of $F$ (i.e., a quadratic extension of $F$ where $\infty$ is inert) satisfying the so-called ``Heegner hypothesis'': all primes dividing $\n$ split in $K$. Under this assumption, if $\cO_K$ is the integral closure of $A$ in $K$ then we can find an ideal $\mathcal N$ of $\cO_K$ such that $\cO_K/\mathcal N\cong A/\n$.  Now fix an ideal $\mathfrak c$ of $A$ with $(\mathfrak c,\n)=1$. We want to recall how to build a certain Drinfeld-Heegner point on $E$ which depends on $\mathfrak c$. To do this, take the order $\cO_{\mathfrak c}:=A+\mathfrak c\cO_K$ in $\cO_K$ of \mbox{conductor $\mathfrak c$}. Since $(\mathfrak c,\n)=1$, the ideal $\mathcal N_\mathfrak c:=\mathcal N\cap\cO_{\mathfrak c}$ is a proper ideal of $\cO_{\mathfrak c}$ with $\cO_\mathfrak c/\mathcal N_\mathfrak c\cong\cO_K/\mathcal N$. It follows that the rank two $A$-lattices $\cO_{\mathfrak c}$ and $\mathcal N^{-1}_{\mathfrak c}$ in $\mathbf C_\infty$ determine a pair $(\Phi_\mathfrak c,\Phi'_\mathfrak c)$ of Drinfeld modules of rank two with a cyclic $\n$-isogeny, hence a point $x_{\mathfrak c}$ on the (compactified) Drinfeld modular curve $X_0(\n)$ of level $\n$. Moreover, as $\text{End}(\Phi)\cong\text{End}(\Phi')\cong\cO_\mathfrak c$, complex multiplication for Drinfeld modules ensures that $x_{\mathfrak c}$ is rational over the ring class field $H_{\mathfrak c}$ of $K$ of conductor $\mathfrak c$. As described in \cite[\S 8]{h}, this field is an abelian extension of $K$ which is unramified outside the primes dividing $\mathfrak c$; furthermore, class field theory provides a canonical isomorphism $\gal{H_{\mathfrak c}}{K}\cong\text{Pic}(\cO_{\mathfrak c})$ (\cite[Theorem 8.8]{h}). Finally, as thoroughly explained in \cite{gr}, there is a non-constant morphism
\[ \pi_E: X_0(\n) \longrightarrow E \]
defined over $F$, i.e. a finite surjective morphism of $F$-schemes. In other words, the elliptic curve $E$ is known to be ``modular'', a property that in the number field case has been established in full generality only for elliptic curves over $\Q$, thanks to the work of Wiles and his school on the Conjecture of Shimura, Taniyama and Weil. We fix such a modular parametrization once and for all, and set $\alpha_{\mathfrak c}:=\pi_E(x_{\mathfrak c})\in E(H_{\mathfrak c})$. The point $\alpha_\mathfrak c$ is the Drinfeld-Heegner point we alluded to before. For more details on these results, which the initiated reader will easily recognize as the function field counterpart of the now-classical constructions of Heegner points in the number field setting as described, e.g., in \cite{g2}, we refer to \cite[\S 2]{bro1}, \cite[Ch. 4]{bro2}.

Now set $G:=\gal{H_\mathfrak c}{K}$, and consider the dual group $\hat{G}:=\text{Hom}(G,\C^\times)$ of complex characters of $G$. Then $G$ acts naturally on the complexified Mordell-Weil group $E(H_\mathfrak c)\otimes_{\mathbb Z} \C$, and we have a decomposition
\[ E(H_\mathfrak c)\otimes\C = \bigoplus_{\chi\in\hat{G}} E(H_\mathfrak c)^\chi \]
as a direct sum of eigenspaces, where
\[ E(H_\mathfrak c)^\chi := \bigl\{x\in E(H_\mathfrak c)\otimes\C \mid \text{$x^\sigma=\chi(\sigma)x$ for all $\sigma\in G$}\bigr\}. \]
Note that $E(K)\otimes\C\subset E(H_\mathfrak c)^{\boldsymbol 1}$ if $\boldsymbol 1$ is the trivial character of $G$. Finally, let
\begin{equation} \label{idempotent-eq}
e_\chi := \frac{1}{|G|}\sum_{\sigma\in G}\chi^{-1}(\sigma)\sigma 
\end{equation}
be the idempotent in the group ring $\Q(\boldsymbol \mu_{|G|})[G]$ giving the projection onto the $\chi$-eigenspace. Set $\alpha_{\mathfrak c,\chi}:=e_\chi(\alpha_\mathfrak c)\in E(H_\mathfrak c)^\chi$. The main result of this note is the following
\begin{thm}  \label{main-thm}
If $\alpha_{\mathfrak c,\chi}\not=0$ then $\dim_{\mathbb C}E(H_\mathfrak c)^\chi=1$.
\end{thm}
This is the function field analogue of a result proved by Bertolini and Darmon (\cite{bd}) for rational elliptic curves (without complex multiplication). The theorem of Bertolini and Darmon was inspired by a conjecture of Gross on the relation between Heegner points and the dimension of eigenspaces of certain Mordell-Weil groups of modular abelian varieties (cf. \cite[Conjecture 11.2]{g1}), and Theorem \ref{main-thm} may certainly be viewed in a broader automorphic-theoretic context, although this issue is not pursued in this paper. Let now $\alpha_K:=\tr{H_{\mathfrak c}}{K}(\alpha_{\mathfrak c})\in E(K)$; by specializing Theorem \ref{main-thm} to the trivial character $\chi=\boldsymbol 1$ we deduce
\begin{cor} \label{main-cor}
If $\alpha_K$ is not a torsion point then $E(K)$ has rank one.
\end{cor}
Corollary \ref{main-cor} is (a weak version of) the main result of \cite{bro2} (cf. \cite[Theorem 1.13.1]{bro2}), and is the function field counterpart of the first major achievement of Kolyvagin's theory of Euler systems of Heegner points on rational elliptic curves (see, e.g., \cite[Theorem 1.3]{g2}). Actually, albeit not evident from Theorem \ref{main-thm}, in the course of our arguments we recover the results of Brown in full strength\footnote{The only possible exception being the explicit determination, in terms of the index of $\alpha_K$ in $E(K)$, of an annihilator of the $\ell$-primary subgroup of the Shafarevich-Tate group of $E_{/K}$.}, in particular (thanks to work by Kato and Trihan) the validity of the Conjecture of Birch and Swinnerton-Dyer for $E_{/F}$ and $E_{/K}$; the reader is referred to Section \ref{trivial-sec} below for precise statements and details. 

Finally we remark that, from our point of view, the proof of Theorem \ref{main-thm} is interesting because it shows that, once one has a good control on the Galois structure of the torsion points of elliptic curves, results of this kind can be obtained in the function field setting too by genuine Kolyvagin-type arguments, without resorting to the highbrow formalism of ``Heegner sheaves'' and ``Heegner modules'' recently introduced by Brown in his monograph \cite{bro2}.\\

\noindent\emph{Notation and conventions.} In the following, let $F^s\subset\bar{F}$ be a separable and an algebraic closure of $F$, respectively, viewed as subfields of $\mathbf C_\infty$ via a fixed embedding. The symbol $\boldsymbol \mu_r$ denotes the $r$th roots of unity in $\bar{F}$. 

For each effective divisor $\mathfrak d$ on the affine curve $\cC-\{\infty\}$ let $H_\mathfrak d$ be the ring class field of $K$ of conductor $\mathfrak d$. For simplicity, set $H:=H_\mathfrak c$ and $\alpha:=\alpha_\mathfrak c$. If $w$ is a place of $H$ then $\F_w$ is the residue field of the completion $H_w$. 

Throughout our work we will not distinguish between effective divisors on $\cC-\{\infty\}$ and ideals of $A$; thus, when dealing with divisors, we will often pass from additive to multiplicative notation (and \emph{vice versa}) according to convenience, without explicit warning. We write $\kappa(z)$ for the (finite) residue field at a closed point $z$ of $\cC$.

If $M$ is an abelian group (e.g., a Galois cohomology group, the geometric points of $E$, etc.) and $t\geq1$ is an integer then $M[t]$ denotes the $t$-torsion subgroup of $M$. In particular, for each integer $m\geq1$ not divisible by $p$ we write $E[m]$ for the $m$-torsion subgroup of $E(\bar{F})$, so that we have a non-canonical isomorphism $E[m]\cong(\Z/m\Z)^2$. 

Finally, we always (often tacitly) assume that $\F_q$ is algebraically closed in $F$ and in the quadratic extension $K$ of $F$, which amounts to asking that $\F_q$ be the field of constants of both $F$ and $K$. This condition is introduced in order to simplify the exposition: dropping it would add annoying technicalities to the proofs, while bringing at the same time no significant novelty to the main arguments.\\

\noindent\emph{Acknowledgements.} I would like to thank Ignazio Longhi for useful discussions and Douglas Ulmer for bringing the paper \cite{lrs} to my attention. 

\section{Galois action on the torsion of $E$} 

We collect here, sometimes without proofs, some results of Igusa (cf. \cite{i}) and other related facts on the Galois module structure of the torsion points of $E$. In our context, these are the analogue in characteristic $p$ of the celebrated ``open image'' theorem of Serre (\cite{s}) for elliptic curves without complex multiplication over number fields. We adopt the notation of \cite[\S 3]{bro1}.

\subsection{Igusa's theorem} \label{igusa-subsec}

Let $n$ be an integer prime to $p=\text{char}(F)$, let $E[n]$ be as above and let $E[p']$ be the prime-to-$p$ part of the torsion subgroup of $E$. There is a Galois representation
\[ \rho_{E,n}:G_F:=\gal{F^s}{F}\longrightarrow\text{Aut}(E[n])\cong GL_2(\Z/n\Z) \]
given by the natural action of $G_F$ on the $n$-torsion points of $E$. Composing $\rho_{E,n}$ with the determinant
\[ \text{det}:\text{Aut}(E[n])\longrightarrow(\Z/n\Z)^\times \]
induces a homomorphism $G_F\rightarrow(\Z/n\Z)^\times$. Set $H_n:=\langle q\rangle\subset(\Z/n\Z)^\times$ for the cyclic subgroup generated by $q$, so that we have an identification $H_n=\gal{\mathbb F_q(\boldsymbol \mu_n)}{\mathbb F_q}$. As in \cite[\S 3]{bro1}, we define the subgroup $\Gamma_n$ of $GL_2(\Z/n\Z)$ via the short exact sequence
\begin{equation} \label{gamma-n-eq}
0\longrightarrow SL_2(\Z/n\Z)\longrightarrow\Gamma_n \xrightarrow{\text{det}}H_n\longrightarrow0. 
\end{equation}
Passing to the inverse limit over all integers $n$ not divisible by $p$ we get an exact sequence of profinite groups
\[ 0\longrightarrow SL_2(\hat{\Z}_{p'})\longrightarrow\hat{\Gamma}\longrightarrow\hat{H}\longrightarrow0. \]
Here $\hat{\Z}_{p'}:=\prod_{\ell\not=p}\Z_\ell$ is the prime-to-$p$ profinite completion of $\Z$, the group $\hat{\Gamma}$ is closed in $GL_2(\hat{\Z}_{p'})$, and $\hat{H}$ is the subgroup of $\hat{\Z}^\times_{p'}$ which is topologically generated by $q$. If we perform the same procedure restricting instead to the powers of a prime $\ell\not=p$, the sequence \eqref{gamma-n-eq} yields a sequence
\[ 0\longrightarrow SL_2(\Z_\ell)\longrightarrow\hat{\Gamma}_\ell\longrightarrow\hat{H}_\ell\longrightarrow0. \]
Since we are assuming that $E_{/F}$ is not isotrivial, we can state 
\begin{thm}[Igusa] \label{igusa-thm}
The profinite group $\gal{F(E[p'])}{F}$ is an open subgroup of $\hat{\Gamma}$.
\end{thm}
If $E[\ell^\infty]$ is the $\ell$-primary part of the torsion of $E$ (for $\ell$ a prime number), Theorem \ref{igusa-thm} can be equivalently formulated as
\begin{thm} \label{igusa2-thm}
The profinite group $\gal{F(E[\ell^\infty])}{F}$ is an open subgroup of $\hat{\Gamma}_\ell$ for all prime numbers $\ell\not=p$, and is equal to $\hat{\Gamma}_\ell$ for almost all such $\ell$'s.
\end{thm}
See \cite[\S 3]{bro1} for a sketch of proof of this result (cf., in particular, \cite[Theorem 3.1]{bro1}). An alternative proof via Tate's theory of analytic uniformization (together with arithmetic applications) can be found in the forthcoming article \cite{blv}.
\begin{rem}
In their paper \cite{ch}, Cojocaru and Hall give a uniform version of Igusa's theorem. More precisely, they show that there exists a positive constant $c(F)$, depending at most on the genus of $\cC$, such that $\gal{F(E[\ell])}{F}\cong\Gamma_\ell$ for any non-isotrivial elliptic curve $E_{/F}$ and any prime number $\ell\geq c(F)$, $\ell\not=p$. Moreover, they determine an explicit expression for $c(F)$: see \cite[Theorem 1]{ch}.
\end{rem} 

\subsection{Torsion points over ring class fields}

Theorem \ref{igusa-thm} will be used in our work in the guise of the following proposition and the subsequent corollary.
\begin{pro} \label{lin-disj-pro}
For almost all prime numbers $\ell\not=p$:
\begin{itemize}
\item[(a)] $\gal{F(E[\ell])}{F}\cong\Gamma_\ell$;
\item[(b)] if $\mathfrak d$ is an effective divisor on $\cC-\{\infty\}$ then 
\begin{equation} \label{lin-disj-F-eq}
H_\mathfrak d\cap F(E[\ell])=F, 
\end{equation}
hence $H_\mathfrak d$ and $F(E[\ell])$ are linearly disjoint over $F$. 
\end{itemize}
\end{pro} 
\begin{proof} Part (a) is an immediate consequence of Theorem \ref{igusa2-thm} (indeed, it suffices to take $\ell\geq c(F)$ with $c(F)$ equal to the uniform constant appearing in the work by Cojocaru and Hall cited above). Let us prove (b). To begin with, let $\ell\not=p$ be a prime and set 
\[ M_{\mathfrak d,\ell}:=H_{\mathfrak d}\cap F(E[\ell]), \qquad M'_{\mathfrak d,\ell}:=H_{\mathfrak d}\cap K(E[\ell]). \]
We will actually show that \eqref{lin-disj-F-eq} follows from the corresponding equality with $K$ in place of $F$. Thus suppose for a moment that 
\begin{equation} \label{lin-disj-K-eq}
M'_{\mathfrak d,\ell}=K.
\end{equation}
Since the extension $K/F$ has degree two, it is clear that $[K(E[\ell]):F(E[\ell])]=1$ or $2$. If $K(E[\ell])=F(E[\ell])$ then $K\subset F(E[\ell])$; but $K/F$ is abelian, hence in particular solvable, and \cite[Corollary 7.3.7]{bro2} says that the only intermediate extensions of $F(E[\ell])/F$ that are solvable over $F$ are those obtained by adjoining $\ell$-power roots of unity to $F$. But if $K$ were built from $F$ in this way then the field of constants $\F_q$ would not be algebraically closed in $K$, and this would contradict our working assumption (cf. Introduction). So we get that necessarily 
\[ \bigl[K(E[\ell]):F(E[\ell])\bigr]=2. \]
The following diagram of field extensions clarifies the situation:
\[ \xymatrix@R=0pt{K(E[\ell])\ar@{-}[dd]_-{2}\ar@{-}[dr] & &\\
                    & M'_{\mathfrak d,\ell}\ar@{=}[dr]\ar@{-}[dd] &\\
                    F(E[\ell])\ar@{-}[dr]& & K\ar@{-}[dd]^-{2}\\
                    & M_{\mathfrak d,\ell}\ar@{-}[dr] &\\
                    & & F} \]   
Restriction of automorphisms gives a natural identification of $\gal{K(E[\ell])}{F(E[\ell])}$ with $\gal{K}{F}$; set $\Delta:=\gal{K(E[\ell])}{F(E[\ell])}=\gal{K}{F}$. Then $F(E[\ell])=K(E[\ell])^\Delta$, and there are equalities
\[ M_{\mathfrak d,\ell}=H_{\mathfrak d}\cap K(E[\ell])^\Delta=\bigl(H_{\mathfrak d}\cap K(E[\ell])\bigr)^\Delta=K^\Delta=F, \]
which is precisely what is claimed in \eqref{lin-disj-F-eq}.

It remains to prove that \eqref{lin-disj-K-eq} holds for all but finitely many primes $\ell$, which we do by arguments similar to those above. Indeed, the elliptic curve $E_{/K}$ is non-isotrivial, and there is a tower of fields
\[ K\subset M'_{\mathfrak d,\ell}\subset K(E[\ell]). \]
Since $\gal{H_{\mathfrak d}}{K}\cong\text{Pic}(\cO_{\mathfrak d})$, the extension $M'_{\mathfrak d,\ell}/K$ is abelian, hence in particular solvable. Then, once again by \cite[Corollary 7.3.7]{bro2}, the field $M'_{\mathfrak d,\ell}$ is obtained from $K$ by adjoining $\ell$-power roots of unity. But the field of constants of $H_{\mathfrak d}$ has degree $d_\infty$ over $\F_q$ (\cite[Theorem 8.10]{h}), and so $M'_{\mathfrak d,\ell}=K$ if $\ell\nmid(q^{d_\infty}-1)$. \end{proof}
\begin{rem}
In the proof of Proposition \ref{lin-disj-pro} we made no use of ramification considerations. This is in contrast with the arguments usually employed for results of this sort in the number field case, especially when the base field is $\Q$. The point is that Minkowski's theorem ensures that there is no unramified extension of $\Q$ of degree $>1$, and this basic fact can be effectively combined with the Criterion of N\'eron-Ogg-Shafarevich which gives the non-ramification of $\Q(E[\ell])/\Q$ outside $N\ell$ if $N$ is the conductor of $E_{/\mathbb Q}$ (cf. \cite[Lemma 2.1]{bd}). Since no result directly analogous to Minkowski's theorem is available for global function fields, in the present context we had to adopt a different strategy. 
\end{rem}
From Proposition \ref{lin-disj-pro} we can easily derive
\begin{cor} \label{no-l-tors-cor}
For almost all prime numbers $\ell\not=p$:
\[ E[\ell](H_\mathfrak d)=\{0\} \]
for any effective divisor $\mathfrak d$ on $\cC-\{\infty\}$. 
\end{cor}
\begin{proof} Let $\ell$ be any of the primes satisfying the statements of Proposition \ref{lin-disj-pro}. By part (a) of that proposition, we know that $\gal{F(E[\ell])}{F}\cong\Gamma_\ell$. Moreover, since $H_\mathfrak d$ and $F(E[\ell])$ are linearly disjoint over $F$ by (b) of the same proposition, $\gal{H_\mathfrak d(E[\ell])}{H_\mathfrak d}\cong\Gamma_\ell$ too. Then observe that $\bigl(\begin{smallmatrix}q&0\\0&q\end{smallmatrix}\bigr)$ is in $\Gamma_\ell$, which means that the multiplication-by-$q$ homothety lies in $\gal{H_\mathfrak d(E[\ell])}{H_\mathfrak d}$. If $P\in E[\ell](H_\mathfrak d)$ then $(q-1)P=0$, whence $P=0$ if $\ell\nmid(q-1)$. \end{proof} 
In conclusion, the elements in the set
\begin{equation} \label{p-set-eq}
\mathcal P:=\bigl\{\text{$\ell$ rational prime $\mid\ell\geq c(F)$, $\ell\not=p$ and $\ell\nmid(q^{d_\infty}-1)$}\bigr\}
\end{equation}
satisfy both Proposition \ref{lin-disj-pro} and Corollary \ref{no-l-tors-cor}.

\section{Preliminary choices and reductions}

As in \cite{bd}, our strategy to prove Theorem \ref{main-thm} will be to do an $\ell$-descent for a suitable prime $\ell$ and then apply Kolyvagin-type techniques, as illustrated for example in \cite{g2}. In this section we gather the basic results that will be repeatedly used in the rest of the paper.

\subsection{The auxiliary prime $\ell$} \label{aux-subsec}

Recall that $p=\text{char}(F)$. If $\Phi_E$ is the group of connected components of the N\'eron model of $E$ over $A$, choose a prime $\ell\not=p$ such that:
\begin{itemize}
\item[1.] $\gal{F(E[\ell])}{F}\cong\Gamma_\ell$;
\vskip 2mm
\item[2.] $H_\mathfrak d\cap F(E[\ell])=F$ for all effective divisors $\mathfrak d$ on $\cC-\{\infty\}$;
\vskip 2mm
\item[3.] $\ell\nmid|H^1(\F_w,\Phi_E)|$ for all primes $w$ of $H$ dividing $\n$;
\vskip 2mm
\item[4.] $\ell\equiv1\pmod{|G|}$;
\vskip 2mm
\item[5.] $\bigl(\begin{smallmatrix}1&0\\0&-1\end{smallmatrix}\bigr)\in\Gamma_\ell$.
\end{itemize} 
The next proposition addresses the question of the compatibility of the conditions above; as we will see, its proof boils down to a judicious application of \v{C}ebotarev's density theorem. However, note that, due to the fact that Igusa's results on the image of Galois in positive characteristic are intrinsically less strong than Serre's ones in characteristic zero, our work will be slightly harder than the one needed to achieve similar compatibility statements for elliptic curves (without CM) over number fields (which are an immediate consequence of the combination of Serre's results and Dirichlet's theorem on primes in arithmetic progressions).
\begin{pro} \label{compatibility-pro}
Conditions $1$-$5$ hold simultaneously for infinitely many primes $\ell$.
\end{pro}
\begin{proof} In light of Proposition \ref{lin-disj-pro}, a moment's thought reveals that the critical point is the compatibility of $4$ and $5$. Since we know that the first three conditions exclude only a finite number of primes, it suffices to show that $4$ and $5$ are simultaneously verified for infinitely many primes $\ell$. To begin with, write $|G|=2^tN$ with $N$ odd, and if $q=p^m$ let $2^r$ be the highest power of $2$ dividing $m$. Then if we set $L'_s:=\Q(\boldsymbol \mu_{2^s},q^{1/2^s})$ we get (essentially by Kummer's theory) that $[L'_s:\Q(\boldsymbol \mu_{2^s})]=2^{s-r}$ for each $s\geq r$ (cf. also \cite[p. 350]{bro2}). It follows that
\begin{equation} \label{cyc-deg-eq}
[L'_s:\Q]=[L'_s:\Q(\boldsymbol \mu_{2^s})]\cdot[\Q(\boldsymbol \mu_{2^s}):\Q]=2^{2s-r-1}.
\end{equation}
Moreover, if $L_s:=\Q(\boldsymbol \mu_{2^sN})$ then
\begin{equation} \label{cyc-deg-eq2}
[L_s:\Q]=2^{s-1}\phi(N)  
\end{equation}
where $\phi$ is Euler's classical function. Let $s$ be the smallest integer such that $s\geq t$ and $[L'_s:\Q]>[L_s:\Q]$ (this is just to make a choice, since any integer with these properties would do equally well): such an $s$ exists because of \eqref{cyc-deg-eq} and \eqref{cyc-deg-eq2}. By \v{C}ebotarev's density theorem, the set $\mathcal P_s$ of rational primes that split completely in $L_s$ and the set $\mathcal P'_s$ of primes that do not split completely in $L'_s$ have (Dirichlet) densities
\begin{equation} \label{densities-eq}
d(\mathcal P_s)=\frac{1}{[L_s:\Q]}, \qquad d(\mathcal P'_s)=1-\frac{1}{[L'_s:\Q]}.
\end{equation} 
The crucial remark is that the set $\mathcal P''_s:=\mathcal P_s\cap\mathcal P'_s$ cannot be finite, for otherwise \eqref{densities-eq} would imply that
\[ d(\mathcal P'_s\cup\mathcal P_s)=d(\mathcal P'_s)+d(\mathcal P_s)=1-\frac{1}{[L'_s:\Q]}+\frac{1}{[L_s:\Q]}>1, \]
which is of course impossible. We claim that the infinite set of primes $\mathcal P''_s$ fulfills our needs. To see why this is true, let $\ell\in \mathcal P''_s$. First of all, since $s\geq t$ and $\ell$ splits completely in $L_s$ it follows that 
$\ell$ splits completely in $\Q(\boldsymbol \mu_{|G|})$ as well, i.e. $\ell\equiv1\pmod{|G|}$, which is condition $4$. Thus it remains to check that $\ell$ satisfies $5$. With notation as in \S \ref{igusa-subsec}, $\bigl(\begin{smallmatrix}1&0\\0&-1\end{smallmatrix}\bigr)\in\Gamma_\ell$ if and only if $-1\in H_\ell$. Furthermore, if $2^n\divexact(\ell-1)$ then it is easy to see that $-1\in H_\ell$ precisely when $q$ is a $2^n$th power non-residue modulo $\ell$, i.e. when the congruence 
\begin{equation} \label{cong-q-eq}
x^{2^n}\equiv q\pmod{\ell}
\end{equation}
has no integer solution. Now, we know that $2^s|(\ell-1)$ because $\ell$ splits completely in $L_s$; on the other hand, $\ell$ does not split completely in $L'_s$, and by a well-known result in elementary algebraic number theory this means that the polynomial $x^{2^s}-q$ does not split into linear factors modulo $\ell$\footnote{In a somewhat fancier language, this is also a consequence of Frobenius' density theorem (which is a manifestation of \v{C}ebotarev's density theorem, cf. \cite[p. 32]{sl}).}. Since $\F_\ell$ contains the $2^s$th roots of unity, this last statement is in turn equivalent to the fact that the congruence
\[ x^{2^s}\equiv q\pmod{\ell} \]
has no solution. But $n\geq s$, hence congruence \eqref{cong-q-eq} has no solution as well. \end{proof} 
\begin{rem}
The interest in having at our disposal a Galois element acting as $\bigl(\begin{smallmatrix}1&0\\0&-1\end{smallmatrix}\bigr)$ should come as no surprise; in fact, if $E$ is an elliptic curve over a number field then complex conjugation acts on $E[\ell]$ with eigenvalues $\pm1$, so condition $5$ simply guarantees the existence of an element which behaves ``like complex conjugation''.
\end{rem}
Keep the notation introduced in the proof of Proposition \ref{compatibility-pro}, recall the definition of the set $\mathcal P$ given in \eqref{p-set-eq} and let $\tilde{\ell}$ be the largest prime dividing $|H^1(\F_w,\Phi_E)|$ for some prime $w$ of $H$ such that $w|\n$. In order to make the choice of our auxiliary prime somewhat more explicit, we choose $\ell$ in the infinite set
\begin{equation} \label{tilde-p-set-eq}
\tilde{\mathcal P}:=\bigl\{\text{$\ell$ rational prime $\mid \ell\in\mathcal P\cap\mathcal P''_s$ and $\ell>\tilde{\ell}$}\bigr\}.
\end{equation}
Note that all primes in $\tilde{\mathcal P}$ are odd and do not divide $q-1$. The requirement that $\ell>\tilde{\ell}$ (which implies condition $3$ above) will be crucially used in the proof of Proposition \ref{kol-loc-pro}. Now we can introduce our strategy to prove Theorem \ref{main-thm}.

\subsection{A restatement of the problem} 

We proceed as in \cite{bd} and reformulate Theorem \ref{main-thm} in terms of $\F_\ell$-valued characters. To begin with, we fix an immersion $\boldsymbol \mu_{|G|}\hookrightarrow\F^\times_\ell$, which exists since $|G|$ divides $\ell-1$ by condition $4$ above. In other words, we choose an identification of $\boldsymbol \mu_{|G|}$ with the unique subgroup of $\F^\times_\ell$ of order $|G|$. This allows us to identify complex and $\F_\ell$-valued characters of $G$, and gives a ``reduction'' map $\Z[\boldsymbol \mu_{|G|}]\rightarrow\F_\ell$.  
 
By condition $4$ above the prime $\ell$ does not divide $|G|$, hence Maschke's theorem ensures that any $\F_\ell[G]$-module $M$ decomposes as a direct sum 
\begin{equation} \label{rep-decomp-eq}
M=\bigoplus_{\gamma\in\hat{G}}M^\gamma 
\end{equation}
of primary representations, with $\gamma$ running through the $\F_\ell$-valued characters of $G$. An explicit expression of the splitting \eqref{rep-decomp-eq} is given by $m\mapsto(e_\gamma(m))_{\gamma\in\hat{G}}$ where $e_\gamma$ is the idempotent defined as in \eqref{idempotent-eq}.

Now notice that, with the obvious identifications, if $\alpha_\chi$ is non-zero in $E(H)\otimes\C$ then it is also non-zero in $E(H)\otimes\F_l$ for almost all primes $l$. On the other hand, the prime $\ell$ is chosen so that $E[\ell](H)=\{0\}$, hence 
\[ \dim_{\mathbb C}E(H)^\chi = {\dim}_{\mathbb F_\ell}\bigl(E(H)\otimes\F_\ell\bigr)^\chi. \]
We can thus formulate the ``mod $\ell$'' analogue of Theorem \ref{main-thm}.
\begin{thm} \label{main2-thm}
If $\alpha_\chi\not=0$ in $E(H)\otimes\F_\ell=E(H)/\ell E(H)$ then $\dim_{\mathbb F_\ell}\bigl(E(H)/\ell E(H)\bigr)^\chi=1$.
\end{thm}  
The rest of this paper will be devoted to the proof of Theorem \ref{main2-thm}, which will yield as a consequence Theorem \ref{main-thm} in its original form. 

\subsection{A compatible family of involutions} \label{involutions-subsec}

In the number field case, to every choice of an embedding $\bar{\Q}\hookrightarrow\C$ corresponds a complex conjugation in $\gal{\bar{\Q}}{\Q}$, which in turn induces (possibly trivial)  automorphisms in finite quotients of $\gal{\bar{\Q}}{\Q}$. As in \cite{bro2}, we want to build a family of Galois elements which play the role of complex conjugation in our function field setting.

First of all, recall that we are assuming that $\infty$ is inert in $K$; then, by class field theory, it follows that $\infty$ splits completely in all the extensions $H_\mathfrak d/K$ where $H_\mathfrak d$ is the ring class field of $K$ of conductor $\mathfrak d$, with $\mathfrak d$ not divisible by $\infty$ (\cite[Theorem 8.8]{h}). Hence if $\infty_{\mathfrak d}$ is any prime of $H_\mathfrak d$ above $\infty$ we have a natural identification $H_{\mathfrak d,\infty_{\mathfrak d}}=K_\infty$, which allows us to view $H_\mathfrak d$ as a subfield of $K_\infty$. Choose a coherent family $\{\infty_{\mathfrak d}\}$ of such primes, i.e. choose a prime $\infty'$ of the integral closure of $A$ in $F^s$ above $\infty$ and set $\infty_{\mathfrak d}:=\infty'\cap\cO_{H_{\mathfrak d}}$ where $\cO_{H_{\mathfrak d}}$ is the integral closure of $A$ in $H_{\mathfrak d}$. Let now $\tau_\infty$ be the non-trivial element of $\gal{K_\infty}{F_\infty}=\gal{K}{F}$. By condition $5$ in \S \ref{aux-subsec}, we know that in $\gal{F(E[\ell])}{F}$ we can find an element acting as $\bigl(\begin{smallmatrix}1&0\\0&-1\end{smallmatrix}\bigr)$ on $F(E[\ell])$: denote it by $\tau$. For each effective divisor $\mathfrak d$ on the open curve  $\cC-\{\infty\}$ (i.e., for each ideal $\mathfrak d$ of $A$) let $\tau_\mathfrak d$ be the unique element of $\gal{H_\mathfrak d(E[\ell])}{F}$ such that 
\begin{itemize}
\item $\tau_{\mathfrak d}|_{F(E[\ell])}=\tau$;
\item $\tau_{\mathfrak d}|_{H_\mathfrak d}=\tau_\infty|_{H_\mathfrak d}$.
\end{itemize}
Equivalently, $\tau_{\mathfrak d}=(\tau_\infty|_{H_\mathfrak d},\tau)$ in $\gal{H_\mathfrak d(E[\ell])}{F}=\gal{H_{\mathfrak d}}{F}\times\gal{F(E[\ell])}{F}$. By construction, the elements $\tau_{\mathfrak d}$ are compatible in the obvious sense. In a similar manner we can also define an automorphism $\tau_K\in\gal{K(E[\ell])}{F}$, and $\tau_{\mathfrak d}|_{K(E[\ell])}=\tau_K$ for all $\mathfrak d$ as above. All the elements $\tau_\mathfrak d$, $\tau_K$ are involutions, i.e. have order two in the corresponding Galois groups. Furthermore, the finite extension $H_{\mathfrak d}$ of $K$ is generalized dihedral over $F$, i.e.
\begin{equation} \label{gen-dihed-eq}
\tau_{\mathfrak d}\sigma\tau^{-1}_{\mathfrak d}=\sigma^{-1}
\end{equation}
for all $\sigma\in\gal{H_{\mathfrak d}}{K}$. For a proof of this fact, see \cite[Proposition 2.5.7]{bro2}. 

Quite generally, let now $M$ be a module on which an involution $\eta$ acts. If multiplication by $2$ is invertible in $\text{End}(M)$ then we have a decomposition
\begin{equation} \label{M-decomp-eq}
M=M^+\oplus M^-, 
\end{equation}
where $M^{\pm}:=\{m\in M\mid \eta(m)=\pm m\}$ are the eigenspaces for $\eta$. As a piece of notation, if $x\in M$ and $X\subset M$ we also write
\begin{equation} \label{x-X-eq}
x^{\pm}:=\frac{x\pm\eta(x)}{2}, \qquad X^{\pm}:=\{x^{\pm}\mid x\in X\}.
\end{equation} 
In our applications, the module $M$ will be a cohomology group of $E$ or a Galois group and $\eta$ will be one of the involutions $\tau_{\star}$ described above.

\section{Drinfeld-Heegner points and Kolyvagin classes} 

\subsection{The Euler system} \label{euler-subsec}

Let $\mathfrak d$ be an effective divisor on $\cC-\{\infty\}$ coprime with $\n$. With the notation of the introduction, $\mathcal N_{\mathfrak d}:=\mathcal N\cap\cO_\mathfrak d$ is a proper $\cO_\mathfrak d$-ideal such that $\cO_\mathfrak d/\mathcal N_{\mathfrak d}\cong\cO_K/\mathcal N$. Let $\Phi_\mathfrak d$, resp. $\Phi'_\mathfrak d$, be the Drinfeld $A$-module of rank two over $\mathbf C_\infty$ associated with the $A$-lattice of rank two $\cO_\mathfrak d$, resp. $\mathcal N^{-1}_\mathfrak d$, in $\mathbf C_\infty$. Then $\Phi_\mathfrak d$ and $\Phi'_\mathfrak d$ have general characteristic (i.e., in the notation of \cite[\S 1.1]{bs}, the corresponding ring homomorphisms $A\rightarrow\mathbf C_\infty\{\tau\}\rightarrow\mathbf C_\infty$ are injective) and have complex multiplication by $\cO_\mathfrak d$. Hence the theory of complex multiplication for Drinfeld modules tells us that the point $x_{\mathfrak d}$ on $X_0(\n)$ determined by the pair $(\Phi_\mathfrak d,\Phi'_\mathfrak d)$ is rational over the ring class field $H_{\mathfrak d}$ of $K$ of conductor $\mathfrak d$. The point $x_\mathfrak d$ is a Drinfeld-Heegner point of conductor $\mathfrak d$ on $X_0(\n)$, and its image 
\[ \alpha_\mathfrak d:=\pi_E(x_\mathfrak d)\in E(H_\mathfrak d) \]
is a Drinfeld-Heegner point (or simply Heegner point, for short) of conductor $\mathfrak d$ on $E$. For a nice introduction to various aspects of the theory of Drinfeld modules the reader is referred to \cite{bs}, while a summary of complex multiplication for Drinfeld modules of rank two can be found in \cite[\S 2.5]{bro2}.

Now, if $\ell\not= p$ is a rational prime (not necessarily equal to the auxiliary prime fixed in \S \ref{aux-subsec}) let $T_\ell(E)$ be the $\ell$-adic Tate module of $E$, so that $T_\ell(E):=\varprojlim E[\ell^n]$ with the inverse limit being taken with respect to the multiplication-by-$\ell$ maps. Equivalently, $T_\ell(E)$ is the $\Q_\ell$-dual of the $\ell$-adic cohomology group $H^1(E_{/F^s},\Q_\ell)$. If $z$ is a closed point of $\cC-\{\infty\}$ not contained in the support of $\n$ (i.e., $z$ is a maximal ideal of $A$ not dividing the ideal $\n$) let $\frob_z$ be the conjugacy class of a Frobenius element at $z$ in $\gal{F^s}{F}$, which is well defined only up to elements in the inertia subgroup $I_z$ above $z$. Then set
\[ a_z:=\text{Tr}\bigl(\frob_z|T_\ell(E)^{I_z}\bigr)\in\Z, \]
the trace of the Frobenius at $z$ acting on (the unramified part of) the $\ell$-adic Tate module of $E$. It is well known that the integer $a_z$ does not depend on the prime $\ell\not=p$. 

Let now $\ell$ be our auxiliary prime (cf. \S \ref{aux-subsec}). With a terminology analogous to that of \cite[Definition 3.1]{bd}, we say that a prime $z$ of $F$ is \emph{special} if 
\begin{itemize}
\item $z\nmid \m\cdot{\mathfrak c}\cdot\text{disc}(K)$;
\item $\frob_z(K(E[\ell])/F)=[\tau_K]$,
\end{itemize}
the latter being an equality of conjugacy classes in $\gal{K(E[\ell])}{F}$. In particular, a special prime is a prime of good reduction for $E$ which is inert in $K$. Note that \v{C}ebotarev's density theorem guarantees that the set $\mathcal S$ of special primes has positive density, hence is infinite.
\begin{lem} \label{a-cong-lem}
If $z\in\mathcal S$ then
\[ a_z \equiv |\kappa(z)|+1 \equiv 0 \pmod{\ell}. \]
\end{lem}
\begin{proof} Compare the characteristic polynomials of $\frob_z(K(E[\ell])/F)$ and $\tau_K$ acting on $E[\ell]$, as explained in \cite[Lemma 7.11.5]{bro2}. \end{proof}
\begin{notat} 
From now on, to ease the writing, if $\mathfrak d$ is an effective divisor coprime with $\mathfrak c$ we let $H_\mathfrak d$ stand for the ring class field of $K$ of conductor $\mathfrak d+\mathfrak c$, thus avoiding to make the divisor $\mathfrak c$ explicit in the conductor (in other words, we denote $H_{\mathfrak d+\mathfrak c}$ simply by $H_\mathfrak d$). A similar convention will be adopted for Heegner points on $X_0(\n)$ and $E$, and for the involutions $\tau_\star$ of \S \ref{involutions-subsec}.
\end{notat}
Let $\mathfrak d$ denote a square-free product of primes in $\mathcal S$, and let $z$ be another special prime not dividing $\mathfrak d$. As noted a few lines before, $z$ is inert in $K$; moreover, by class field theory any prime of $H_\mathfrak d$ above $z$ is totally ramified in $H_{\mathfrak d+z}$ (cf. \cite[\S 2.3]{bro2}). Fix a prime $z'$ of $H_\mathfrak d$ above $z$, and let $z''$ the unique prime of $H_{\mathfrak d+z}$ above $z'$.
\begin{pro} \label{euler-pro}
The family of Heegner points $\alpha_\star\in E(H_\star)$ satisfies the following properties:
\begin{itemize}
\item[1.] $u\tr{H_{\mathfrak d+z}}{H_\mathfrak d}(\alpha_{\mathfrak d+z})=a_z\alpha_\mathfrak d$, where $u=\begin{cases}1 & \text{if $\mathfrak d\not=0$}\\|\cO^\times_K|\big/|A^\times| & \text{if $\mathfrak d=0$}\end{cases}$;
\item[2.] $\alpha_{\mathfrak d+z}\equiv\frob_{z'}(\alpha_\mathfrak d) \pmod{z''}$.
\end{itemize}
\end{pro}
\begin{proof} As in the number field case, both properties follow from the corresponding facts about the Heegner points $x_{\mathfrak d+z}$ and $x_\mathfrak d$ on $X_0(\n)$. In particular, part $1$ is obtained by comparing the trace and the Hecke correspondence divisors on $X_0(\n)$, while $2$ is proved via the Eichler-Shimura congruence relation for Drinfeld modular curves (\cite[Theorem 4.5.7]{bro2}). The reader may wish to consult \cite[Table 4.8.5 and Theorem 4.8.9]{bro2} for all the omitted details. \end{proof}
In the terminology of Kolyvagin (cf. \cite[\S 3]{g2}), Heegner points on $E$ form an \emph{Euler system}, as in the classical case. Let now $\epsilon$ be the sign in the functional equation of the $L$-function of $E$; the analogue of \cite[Lemma 2.3]{bd} is the following
\begin{pro} \label{tau-heegner-pro}
There exists $\sigma_0\in\gal{H_\mathfrak d}{K}$ such that
\[ \tau_\mathfrak d(\alpha_\mathfrak d)=-\epsilon\sigma_0(\alpha_\mathfrak d) \]
in $E(H_\mathfrak d)/\ell E(H_\mathfrak d)$. Moreover:
\[ \tau_\mathfrak de_\chi(\alpha_\mathfrak d)=-\epsilon\bar{\chi}(\sigma_0)e_{\bar{\chi}}(\alpha_\mathfrak d). \]
\end{pro}
\begin{proof} By \cite[Theorem 4.8.6]{bro2}, there is $\sigma_0\in\gal{H_\mathfrak d}{K}$ such that $\tau_\mathfrak d(\alpha_\mathfrak d)+\epsilon\sigma_0(\alpha_\mathfrak d)$ is a torsion point. But $E[\ell](H_\mathfrak d)=\{0\}$ by Corollary \ref{no-l-tors-cor}, and the first statement follows. As far as the second claim is concerned, it is enough to observe that $\tau_\mathfrak de_\chi=e_{\bar{\chi}}\tau_\mathfrak d$ by equation \eqref{gen-dihed-eq}. \end{proof} 

\subsection{Kolyvagin classes} 

The formal construction of Kolyvagin classes out of the Euler system of Heegner points described in \S \ref{euler-subsec} is completely analogous to the one performed in the number field setting (cf. \cite[\S\S 3,4]{g2}), so we will be brief. More details can be found in \cite[\S 5.6]{bro1} and, somewhat in disguise, in \cite[\S 7.14]{bro2}.

Let $\mathfrak d$ be a square-free product of primes in $\mathcal S$. The field $H_\mathfrak d$ is the compositum of the fields $H_z$ where $z$ runs through the prime divisors of $\mathfrak d$, and these extensions are linearly disjoint over $H$. So, if we write $G_\mathfrak f:=\gal{H_\mathfrak f}{H}$ for any effective divisor $\mathfrak f$ coprime to $\mathfrak c$, we have a canonical identification $G_\mathfrak d=\oplus_{z|\mathfrak d}G_z$. In other words, we identify $G_z$ with the subgroup $\gal{H_\mathfrak d}{H_{\mathfrak d-z}}$ of $G_\mathfrak d$. By class field theory, since $z$ is inert in $K$ the extension $H_z/H$ is cyclic of order
\begin{equation} \label{H-z-deg-eq}
[H_z:H] = \begin{cases} |\kappa(z)|+1 & \text{if $\mathfrak c\not=0$},\\[2mm] |A^\times|\cdot(|\kappa(z)|+1)\big/|\cO^\times_K| & \text{if $\mathfrak c=0$}. \end{cases} 
\end{equation}
See \cite[\S 2.3]{bro2} for a proof of these statements. Since we are assuming that $\F_q$ is algebraically closed in $K$, we have $K\not=F\otimes_{\mathbb F_q}\F_{q^2}$, so $\cO^\times_K=A^\times$. Hence, in Proposition \ref{euler-pro} and equation \eqref{H-z-deg-eq}, the two cases $\mathfrak c\not=0$ and $\mathfrak c=0$ coincide. Set $m_z:=[H_z:H]=|\kappa(z)|+1$, choose for each $z$ a generator $\sigma_z$ of the cyclic group $G_z$ and define
\[ \text{Tr}_z:=\sum_{i=0}^{m_z}\sigma^i_z, \qquad D_z:=\sum_{i=1}^{m_z}i\sigma^i_z, \qquad D_\mathfrak d:=\prod_{z|\mathfrak d}D_z \]
as elements of the group rings $\F_\ell[G_z]$, $\F_\ell[G_z]$ and $\F_\ell[G_\mathfrak d]$, respectively. If $[\star]$ denotes the class of the element $\star$ in a quotient group, we can state
\begin{lem} \label{D-fix-lem}
$[D_\mathfrak d(\alpha_\mathfrak d)]\in\bigl(E(H_\mathfrak d)/\ell E(H_\mathfrak d)\bigr)^{G_\mathfrak d}$.
\end{lem}
\begin{proof} For all $z|\mathfrak d$ a simple computation shows that
\[ (\sigma_z-1)D_z=m_z-\text{Tr}_z=-\text{Tr}_z, \]
the last equality coming from Lemma \ref{a-cong-lem}. Combining Proposition \ref{euler-pro} with Lemma \ref{a-cong-lem} we get:
\[ -\text{Tr}_z(\alpha_\mathfrak d)=-a_z\alpha_{\mathfrak d-z}=0, \]
and the claim follows because the elements $\sigma_z$ generate the group $G_\mathfrak d$. \end{proof}
Now, Corollary \ref{no-l-tors-cor} says that $E[\ell](H_\mathfrak d)=\{0\}$, hence there is a short exact sequence
\[ 0 \longrightarrow E(H_\mathfrak d) \overset{\ell}{\longrightarrow} E(H_\mathfrak d) \longrightarrow E(H_\mathfrak d)/\ell E(H_\mathfrak d) \longrightarrow 0. \]
By taking $G_\mathfrak d$-invariants we get a short exact sequence of $\F_\ell[G]$-modules
\begin{equation} \label{G-inv-eq}
0 \longrightarrow E(H)/\ell E(H) \longrightarrow \bigl(E(H_\mathfrak d)/\ell E(H_\mathfrak d)\bigr)^{G_\mathfrak d} \longrightarrow H^1(G_\mathfrak d,E(H_\mathfrak d))[\ell] \longrightarrow 0.
\end{equation}
We can introduce Kolyvagin classes.
\begin{dfn} \label{kol-class-dfn}
The \emph{Kolyvagin class} $\xi(\mathfrak d)$ is the image of $[D_\mathfrak d(\alpha_\mathfrak d)]$ in the cohomology group $H^1(G_\mathfrak d,E(H_\mathfrak d))[\ell]$ under the map of sequence \eqref{G-inv-eq}.
\end{dfn}
By a notational abuse, we use the same symbol to denote also the image in $H^1(H,E)[\ell]$ of the class $\xi(\mathfrak d)$ via the inflation map $H^1(G_\mathfrak d,E(H_\mathfrak d))[\ell]\hookrightarrow H^1(H,E)[\ell]$.

\subsection{Local properties of the classes $\xi(\mathfrak d)$} 

Let $z$ be a prime of $F$, let $w$ be a prime of $K$ above $z$ and let
\[ \text{res}_w: H \; \longmono \; \bigoplus_{w'|w}H_{w'} \]
be the natural immersion by diagonal embedding. As customary, we extend $\text{res}_w$ in a ``functorial'' way, leaving the context to make the correct interpretation clear. The following result describes the behaviour of $\xi(\mathfrak d)$ under the localization map.
\begin{pro} \label{kol-loc-pro}
\begin{itemize}
\item[1.] If $z\nmid\mathfrak d$ then $\mathrm{res}_w\xi(\mathfrak d)=0$;
\item[2.] If $z\in\mathcal S$ then there is a canonical $G$-equivariant isomorphism
\[ T: \bigoplus_{z'|z}H^1(H_{z'},E)[\ell] \overset{\cong}{\longrightarrow} \bigoplus_{z'|z}\mathrm{Hom}\bigl(\boldsymbol \mu_\ell(\F_{z'}),\tilde{E}[\ell](\F_{z'})\bigr) \]
such that when $z|\mathfrak d$ the homomorphism $T\bigl(\mathrm{res}_z\xi(\mathfrak d)\bigr)$ maps each group $\boldsymbol \mu_\ell(\F_{z'})$ onto the subgroup of $\tilde{E}[\ell](\F_{z'})$ generated by the element
\[ \left(\frac{m_z\mathrm{Frob}_z-a_z}{\ell}\right)D_{\mathfrak d-z}(\alpha_{\mathfrak d-z}). \]
\end{itemize}
Here $\tilde{E}$ is the reduction modulo $z'$ of $E$.
\end{pro}
In Proposition \ref{kol-loc-pro}, which is the counterpart of \cite[Lemma 3.4]{bd}, we identify $D_{\mathfrak d-z}(\alpha_{\mathfrak d-z})$ with its reduction $\bmod\:z'$.  
\begin{proof} 1. If $z\nmid\mathfrak d$ then $\xi(\mathfrak d)$ is the image via inflation of a cohomology class relative to the extension $H_\mathfrak d/H$. This extension is unramified in $w'$, hence $\text{res}_w\xi(\mathfrak d)$ belongs to $H^1(H^{nr}_{w'}/H_{w'},E)$, where $H^{nr}_{w'}$ is the maximal unramified extension of $H_{w'}$. But this group is trivial when $E$ has good reduction at $z$ (cf. \cite[Ch. 1, \S 3]{m}), so $\text{res}_w\xi(\mathfrak d)=0$ if $z\nmid\n$.

Let $\E$ be the N\'eron model of $E$ over $A$ and let $\E^0$ be the connected component of the identity in $\E$, so that the finite quotient $\Phi_E:=\E/\E^0$ is the group of connected components of $\E$. In the case of bad reduction, i.e. when $z|\n$, we have $H^1(H^{nr}_{w'}/H_{w'},\E^0)=0$, so there is an injection $H^1(H^{nr}_{w'}/H_{w'},E)\hookrightarrow H^1(\F_{w'},\Phi_E)$ (cf. \cite[Ch. 1, Proposition 3.8]{m}). But we required that $\ell\nmid|H^1(\F_{w'},\Phi_E)|$ (cf. condition $3$ in \S \ref{aux-subsec}, which is guaranteed by the fact that $\ell$ is chosen in the set $\tilde{\mathcal P}$ defined in \eqref{tilde-p-set-eq}), whence $\text{res}_w\xi(\mathfrak d)=0$ since $\text{res}_w\xi(\mathfrak d)$ lies in $H^1(H_{w'},E)[\ell]$.

2. The proof goes \emph{mutatis mutandis} as in the number field case, and we refer to \cite[Proposition 6.2]{g2} for details. An in-depth discussion can be found in \cite[Lemma 7.14.14]{bro2}. \end{proof}
\begin{rem}
Condition $3$ that we required of $\ell$ in \S \ref{aux-subsec} allowed us to avoid, in the proof of the bad reduction case in part 1 of Proposition \ref{kol-loc-pro}, the somewhat delicate arguments involving the arithmetic of the jacobian variety of $X_0(\n)$ that appear in the proofs of \cite[Proposition 6.2]{g2} and \cite[Lemma 3.4]{bd}.
\end{rem}
In the sequel we will use the following 
\begin{cor} \label{isom-loc-cor}
There is a $G$-equivariant isomorphism
\[ \bigoplus_{z'|z}H^1(H_{z'},E)[\ell] \overset{\cong}{\longrightarrow} \bigoplus_{z'|z}\tilde{E}(\F_{z'})/\ell\tilde{E}(\F_{z'}) \]
which sends $\mathrm{res}_z\xi(\mathfrak d)$ to $\mathrm{res}_z D_{\mathfrak d-z}(\alpha_{\mathfrak d-z})$.
\end{cor} 
Notice that the image of the point $\text{res}_z D_{\mathfrak d-z}(\alpha_{\mathfrak d-z})$ is well defined in $\bigoplus_{z'|z}\tilde{E}(\F_{z'})/\ell\tilde{E}(\F_{z'})$. To deduce Corollary \ref{isom-loc-cor} from part 2 of Proposition \ref{kol-loc-pro}, and to build the isomorphism in an explicit way, see the explanation immediately after \cite[Corollary 3.5]{bd}.

\section{Generators and relations in the dual Selmer group}

As in \cite{bd} (cf. also \cite{g2}), the basic idea is to use Tate's duality to build elements in the dual of a suitable $\ell$-Selmer group, and then to control the ``size'' of this $\F_\ell$-vector space via the Kolyvagin classes introduced in Definition \ref{kol-class-dfn}.

\subsection{Local Poitou-Tate duality: generators} \label{local-duality-subsec}

As before, let $z\in\mathcal S$ and let $z'$ be a prime of $H$ above $z$.  Poitou-Tate duality says that the cup-product 
\[ \langle\,,\rangle_{z'}: H^1(H_{z'},E[\ell])\times H^1(H_{z'},E[\ell]) \longrightarrow \F_\ell \]
is an alternating, perfect pairing of $\F_\ell$-vector spaces (cf. \cite[Ch. 1, Corollary 2.3]{m}). Moreover, the subgroup $E(H_{z'})/\ell E(H_{z'})$ is isotropic for $\langle\,,\rangle_{z'}$, and since $E$ has good reduction at $z'$ there is an induced alternating, perfect pairing
\[ \langle\,,\rangle_{z'}: E(H_{z'})/\ell E(H_{z'})\times H^1(H_{z'},E)[\ell] \longrightarrow \F_\ell \] 
(see \cite[Theorem 7.15.6]{bro2}). This allows us to identify $\bigoplus_{z'|z}H^1(H_{z'},E)[\ell]$ with the $\F_\ell$-dual ${\bigl(\bigoplus_{z'|z}E(H_{z'})/\ell E(H_{z'})\bigr)}^\ast$. Recall the $\ell$-Selmer group of $E$ over $H$ (an analogous definition can be given for any finite extension of $F$):
\[ \begin{split}
   \sel{\ell}{E/H}&:=\ker\Big(H^1(H,E[\ell])\longrightarrow\displaystyle{\prod_v} H^1(H_v,E)[\ell]\Big)\\
                  &\;=\bigl\{s\in H^1(H,E[\ell])\mid\text{$\text{res}_v(s)\in E(H_v)/\ell E(H_v)$ for all places $v$ of $H$}\bigr\}.
   \end{split} \]
Thus we have a natural map
\[ \text{res}_z:\sel{\ell}{E/H}\longrightarrow\bigoplus_{z'|z}E(H_{z'})/\ell E(H_{z'}); \]
passing to the $\F_\ell$-duals and using the identification provided by the local Poitou-Tate pairing we obtain a homomorphism
\begin{equation} \label{psi-z-eq}
\Psi_z:\bigoplus_{z'|z}H^1(H_{z'},E)[\ell]\longrightarrow\dsel{\ell}{E/H}. 
\end{equation}
Denote by $X_z$ the image of $\Psi_z$ in $\dsel{\ell}{E/H}$. Choose an auxiliary $z_1\in\mathcal S$ and define the following Galois extensions of $F$:
\[ L:=H_{z_1}(E[\ell]),\quad M_0:=L(\alpha/\ell)^\text{Gal},\quad M_1:=L\bigl(D_{z_1}(\alpha_{z_1})/\ell\bigr)^\text{Gal}, \quad M:=M_0\cdot M_1. \] 
Here we are fixing arbitrary $\ell$th roots of $\alpha$ and $D_{z_1}(\alpha_{z_1})$ in $E(\bar{F})$, but the fields $M_0$ and $M_1$ do not, of course, depend on the actual choices. Moreover, the superscript ``Gal'' means that we are taking Galois closure over $F$. By conditions $1$ and $2$ that we imposed on $\ell$ in \S \ref{aux-subsec}:
\[ \gal{L}{F}=\gal{H_{z_1}}{F}\times\gal{F(E[\ell])}{F}\cong\gal{H_{z_1}}{F}\times\Gamma_\ell. \]
The Galois groups $V_0$, $V_1$ and $V$ of $M_0$, $M_1$ and $M$, respectively, over $L$ are $\F_\ell$-vector spaces which are equipped with a natural action of $\gal{L}{F}$. The field diagram picturing this situation is the same as the one appearing at the bottom of \cite[p. 69]{bd}. We want to lift the involution $\tau_{z_1}\in\gal{L}{F}$ to an involution in $\gal{M}{F}$; to do this, we need an elementary result in group theory.
\begin{lem} \label{lift-lem}
Let $G$ be a group and let $N$ be a finite normal subgroup of $G$ of odd order. Then an involution $\sigma$ in $G/N$ can be lifted to an involution $\Sigma$ in $G$.
\end{lem}
\begin{proof} Let $\pi:G\rightarrow G/N$ be the canonical quotient map and let $\Sigma'$ be any element of $G$ such that $\pi(\Sigma')=\sigma$. This condition implies that $\Theta:=(\Sigma')^2$ is in $N$. Now set $\Sigma:=(\Sigma')^n$ where $n:=|N|$. Then $\Sigma^2=(\Sigma')^{2n}=\Theta^n=1$, and $\pi(\Sigma)=\pi(\Sigma')^n=\sigma^n=\sigma$ because $2\nmid n$. \end{proof}
To apply Lemma \ref{lift-lem} in our situation we can proceed as follows. Let 
\[ T:=\langle\alpha,D_{z_1}(\alpha_{z_1})\rangle\subset H^1(H_{z_1},E[\ell]) \]
be the subgroup of $H^1(H_{z_1},E[\ell])$ generated by the images of the two points indicated via the coboundary map. The group $T$ is finite because it is contained in the Selmer group $\sel{\ell}{E/H_{z_1}}$, which is finite. An argument analogous to the one in \cite[Corollary 7.18.10]{bro2} shows that $V\cong\text{Hom}(T,E[\ell])$, hence
\begin{equation} \label{m-l-deg-eq}
[M:L]\,\big|\,\ell^{2|T|}. 
\end{equation}
Since $\ell\not=2$, \eqref{m-l-deg-eq} implies, in particular, that the finite extension $M/L$ has odd degree. By applying Lemma \ref{lift-lem} with $G=\gal{M}{F}$, $N=V$ and $\sigma=\tau_{z_1}$ we deduce the existence of the required extension of $\tau_{z_1}$. Thus choose an involution $\tau_M\in\gal{M}{F}$ lifting $\tau_{z_1}$. The involution $\tau_\infty$ acts on $V$ by conjugation by $\tau_M$, and this action does not depend on the particular lifting of $\tau_{z_1}$. Given a subset $U$ of $V$, we follow once again \cite{bd} and define
\[ \mathcal S(U):=\bigl\{\text{$z$ prime of $F$} \mid \text{$\text{Frob}_z(M/F)=[\tau_M u]$ for $u\in U$}\bigr\}. \] 
Before proving the main result of this $\S$ we need two more lemmas.
\begin{lem} \label{vanishing-cohom-lem}
$H^1(\gal{H(E[\ell])}{H},E[\ell])=\mathrm{Hom}_{\mathrm{Gal}(H(E[\ell])/H)}\bigl(\gal{L}{H(E[\ell])},E[\ell]\bigr)=0$. 
\end{lem}
\begin{proof} First of all, since $H$ and $F(E[\ell])$ are linearly disjoint over $F$ by condition $2$ of \S \ref{aux-subsec} we know that $\gal{H(E[\ell])}{H}=\gal{F(E[\ell])}{F}\cong\Gamma_\ell$. The vanishing of the first group is then a particular case of \cite[Proposition 7.3.1]{bro2}. Now notice that
\[ \gal{L}{H(E[\ell])} \; \longmono \; \gal{H_{z_1}}{H} \cong \Z/m_{z_1}\Z, \]
hence $\gal{L}{H(E[\ell])}$ is a cyclic group, say of order $t$ (actually, the inclusion above is an isomorphism, but we will not need this fact). Moreover, the subgroup $\gal{L}{H(E[\ell])}$ is contained in the centre of $\gal{L}{H}$, so $\gal{H(E[\ell])}{H}$ acts trivially on $\gal{L}{H(E[\ell])}$. Thus we must show that 
\[ \text{Hom}_{\Gamma_\ell}(\Z/t\Z,\F^2_\ell)=0 \]
where $\Z/t\Z$ is a trivial $\Gamma_\ell$-module. But this is easy: if $f\in\text{Hom}_{\Gamma_\ell}(\Z/t\Z,\F^2_\ell)$ then  in particular $(q-1)f(1)=0$ (recall that $\bigl(\begin{smallmatrix}q&0\\0&q\end{smallmatrix}\bigr)\in\Gamma_\ell$), whence $f(1)=0$ because $\ell\nmid(q-1)$ since $\ell$ belongs to the set $\tilde{\mathcal P}$ defined in \eqref{tilde-p-set-eq}. \end{proof}
The reader may wish to consult \cite[Proposition 7.3.1]{bro2} for more general vanishing results in Galois cohomology.
\begin{lem} \label{no-inv-sub-lem}
If $\Gamma$ is a subgroup of $GL_2(\F_\ell)$ containing $SL_2(\F_\ell)$ then $\F^2_\ell$ has no non-trivial proper $\Gamma$-invariant subspace.
\end{lem}
\begin{proof} Let $\langle\bigl(\begin{smallmatrix}x\\y\end{smallmatrix}\bigr)\rangle$ be a proper $\Gamma$-invariant subspace of $\F^2_\ell$. Since $\bigl(\begin{smallmatrix}1&1\\0&1\end{smallmatrix}\bigr)\in\Gamma$, there exists an integer $m$ such that $0\leq m\leq\ell-1$ and
\[ \binom{x+y}{y}=\begin{pmatrix}1&1\\0&1\end{pmatrix}\cdot\binom{x}{y}=m\binom{x}{y}. \]
But then $y=0$. Finally, since $\bigl(\begin{smallmatrix}0&-1\\1&0\end{smallmatrix}\bigr)\in\Gamma$ we similarly obtain that $x=0$. \end{proof}
At this point we can state
\begin{pro} \label{gen-pro}
If $U^+$ generates $V^+$ then the $X_z$ with $z$ ranging over $\mathcal S(U)$ generate $\dsel{\ell}{E/H}$.
\end{pro} 
\begin{proof} Now that all the necessary ingredients have been collected, our proof follows that of \cite[Proposition 4.1]{bd} closely. For notational convenience, let $\ddsel{\ell}{E/H}$ denote the dual of $\dsel{\ell}{E/H}$, so that there is a canonical identification $\ddsel{\ell}{E/H}=\sel{\ell}{E/H}$. Thus we can equivalently prove that if $s\in\sel{\ell}{E/H}$ is such that $\Psi_z(\xi)(s)=0$ for all $z\in\mathcal S(U)$ and all $\xi\in\bigoplus_{z'|z}H^1(H_{z'},E)[\ell]$ then $s=0$. Now note that, up to the identification induced by the Poitou-Tate pairing, the map
\[ \Psi_z:\bigoplus_{z'|z}H^1(H_{z'},E)[\ell]=\Big(\bigoplus_{z'|z}E(H_{z'})/\ell E(H_{z'})\Big)^\ast\xrightarrow{\text{res}^\ast_z}\dsel{\ell}{E/H} \]
is given by
\[ \Psi_z(\xi)(t)=\langle\xi,\text{res}_z(t)\rangle_{z'}. \]
But $\langle\xi,\text{res}_z(s)\rangle_{z'}=0$ for all $\xi\in\bigoplus_{z'|z}H^1(H_{z'},E)[\ell]$ if and only if $\text{res}_z(s)=0$. Therefore we can rewrite the claim of the proposition as follows:
\begin{equation*}
\text{if $s\in\sel{\ell}{E/H}$ is such that $\text{res}_z(s)=0$ for all $z\in\mathcal S(U)$ then $s=0$.} \tag{$\ast$}
\end{equation*}
Thus we are reduced to proving $(\ast)$. Without loss of generality we can assume that $s$ lies in an eigenspace for $\tau_\infty|_H\in\gal{H}{F}$. The restriction map
\begin{equation} \label{res-cohom-eq}
H^1(H,E[\ell])\longrightarrow H^1(L,E[\ell])^{\text{Gal}(L/H)}
\end{equation}
can be written as a composition
\[ H^1(H,E[\ell])\overset{\eta}{\longrightarrow}H^1(H(E[\ell]),E[\ell])^{\text{Gal}(H(E[\ell])/H)}\overset{\beta}{\longrightarrow}H^1(L,E[\ell])^{\text{Gal}(H(E[\ell])/H)}. \]
By the inflation-restriction sequence we have:
\[ \ker(\eta)=H^1\bigl(\gal{H(E[\ell])}{H},E[\ell]\bigr) \]
and
\[ \ker(\beta)=\text{Hom}_{\text{Gal}(H(E[\ell])/H)}\bigl(\gal{L}{H(E[\ell])},E[\ell]\bigr), \]
and both these groups are trivial by Lemma \ref{vanishing-cohom-lem}. This shows that the restriction \eqref{res-cohom-eq} is injective; henceforth we identify $s$ with its image by this map in
\[ H^1(L,E[\ell])^{\text{Gal}(L/H)}\subset\text{Hom}_{\text{Gal}(H(E[\ell])/H)}\bigl(\gal{\bar{M}}{L},E[\ell]\bigr) \]
where $\bar{M}$ is the maximal abelian extension of $L$ whose Galois group is of exponent $\ell$. Choose a minimal Galois extension $\tilde{M}$ of $F$ containing $M$ such that $s$ factors through $\gal{\tilde{M}}{L}$. In other words, we require that in the commutative triangle
\[ \xymatrix@R=25pt@C=0pt{\gal{\bar{M}}{L}\ar[rr]^-s\ar[dr] & & E[\ell]\\
             & \gal{\tilde{M}}{L} \ar[ur]_-{s_0} &} \]
the map $s_0$ is an injection. This implies that the finite extension $\tilde{M}/M$ has odd degree, and so by Lemma \ref{lift-lem} the involution $\tau_M$ in $\gal{M}{F}$ extends to an involution $\tau_{\tilde{M}}$ in $\gal{\tilde{M}}{F}$, which we fix. Let now $x\in\gal{\tilde{M}}{F}$ be such that $x|_M\in U$. By \v{C}ebotarev's density theorem there exists $z\in\mathcal S(U)$ with $\text{Frob}_z(\tilde{M}/F)=[\tau_{\tilde{M}}x]$. But we are assuming that $\text{res}_z(s)=0$ for all $z\in\mathcal S(U)$, hence 
\[ s_0\bigl(\text{Frob}_{\tilde{z}}(\tilde{M}/L)\bigr)=0 \]
for all primes $\tilde{z}$ of $\tilde{M}$ above $z$. On the other hand:
\[ \text{Frob}_{\tilde{z}}(\tilde{M}/L)=(\tau_{\tilde{M}}x)^2=x^{\tau_{\tilde{M}}}x=(x^+)^2 \]
for some $\tilde{z}$ above $z$, hence $s_0(x^+)=0$. Since $U^+$ generates $V^+$ by assumption, this yields the vanishing of $s_0$ on $\gal{\tilde{M}}{L}^+$. It follows that the image of $s$, which equals the image of $s_0$, is contained in an eigenspace of $E[\ell]$ for the action of $\tau$, and in particular it is a proper $\gal{L}{H_{z_1}}$-submodule of $E[\ell]$. But $\gal{L}{H_{z_1}}\cong\Gamma_\ell$ as a consequence of conditions $1$ and $2$ of \S \ref{aux-subsec}, hence $s=0$ by Lemma \ref{no-inv-sub-lem}. \end{proof}

\subsection{Global Poitou-Tate duality: relations}

Choose the prime $z_1\in\mathcal S$ of \S \ref{local-duality-subsec} so that $\text{res}_{z_1}\alpha_\chi\not=0$; this implies that
\begin{equation} \label{res-nonzero-eq}
\text{res}_{z_1}\alpha_{\bar{\chi}}\not=0
\end{equation}
by Proposition \ref{tau-heegner-pro}. To show the existence of such a $z_1$, apply \v{C}ebotarev's density theorem to the extension $H(E[\ell])(\alpha_{\chi}/\ell)/F$, using the assumption (made at the outset) that $\alpha_\chi\not=0$ in $E(H)/\ell E(H)$. By Corollary \ref{isom-loc-cor}, condition \eqref{res-nonzero-eq} guarantees that
\begin{equation} \label{res-nonzero2-eq}
\text{res}_{z_1}e_{\bar{\chi}}\xi(z_1)\not=0.
\end{equation} 
Similarly to what done in \S \ref{local-duality-subsec}, define the fields
\[ M^{\bar{\chi}}_0:=L(\alpha_{\bar{\chi}}/\ell), \quad M^{\bar{\chi}}_1:=L\bigl(e_{\bar{\chi}}D_{z_1}(\alpha_{z_1})/\ell\bigr), \quad M^{\bar{\chi}}:=M^{\bar{\chi}}_0\cdot M^{\bar{\chi}}_1, \]
and write $V^{\bar{\chi}}_0$, $V^{\bar{\chi}}_1$ and $V^{\bar{\chi}}$ for the respective Galois groups over $L$. Note that $V^{\bar{\chi}}_i\cong E[\ell]$ for $i=0,1$ (cf. \cite[Corollary 7.18.10]{bro2}).
\begin{lem} \label{lin-disj-twist-lem}
The extensions $M^{\bar{\chi}}_0$ and $M^{\bar{\chi}}_1$ are linearly disjoint over $L$.
\end{lem}
\begin{proof} The coboundary map
\[ E(H_{z_1})/\ell E(H_{z_1})\longrightarrow\text{Hom}_{\text{Gal}(L/H_{z_1})}(V,E[\ell]) \]
is injective, and since $\gal{L}{H_{z_1}}\cong\Gamma_\ell$ it can be shown that linearly independent homomorphisms cut out linearly disjoint extensions over $L$. This shows that linearly independent points in $E(H_{z_1})/\ell E(H_{z_1})$ give rise to linearly disjoint extensions over $L$; on the other hand, the extension of $L$ cut out by the image of $[x]$ with $x\in E(H_{z_1})$ is precisely $L(x/\ell)$, so to prove the lemma it suffices to show that $\alpha_{\bar{\chi}}$ and $e_{\bar{\chi}}D_{z_1}(\alpha_{z_1})$ are linearly independent over $\F_\ell$. If this were not, we would have a relation of the form
\[ e_{\bar{\chi}}D_{z_1}(\alpha_{z_1})=u\cdot\alpha_{\bar{\chi}}, \quad u\in\F^\times_\ell \]
in $E(H_{z_1})/\ell E(H_{z_1})$. In the short exact sequence
\[ 0\longrightarrow\bigl(E(H)/\ell E(H)\bigr)^{\bar{\chi}}\longrightarrow\bigl(E(H_{z_1})/\ell E(H_{z_1})\bigr)^{G_{z_1},\bar{\chi}}\overset{\delta'}{\longrightarrow}H^1(G_{z_1},E)^{\bar{\chi}}[\ell]\longrightarrow0 \]
we have:
\[ \delta'(\alpha_{\bar{\chi}})=0, \qquad \delta'(e_{\bar{\chi}}D_{z_1}(\alpha_{z_1}))=e_{\bar{\chi}}\xi(z_1). \]
But $e_{\bar{\chi}}\xi(z_1)\not=0$ because $\text{res}_{z_1}e_{\bar{\chi}}\xi(z_1)\not=0$ by \eqref{res-nonzero2-eq}, and we are done. \end{proof}
By Lemma \ref{lin-disj-twist-lem}:
\[ V^{\bar{\chi}}=V^{\bar{\chi}}_0\times V^{\bar{\chi}}_1\cong E[\ell]\times E[\ell]. \] 
Thanks to Proposition \ref{tau-heegner-pro}, and using the identity $\tau_{z_1}D_{z_1}=-D_{z_1}\tau_{z_1}$, we can describe the action of $\tau_{z_1}$ on $V^{\bar{\chi}}$. As in \cite[\S 5]{bd}, we need to distinguish between two cases:\\

\emph{Case 1:} $\chi=\bar{\chi}$. A direct computation shows that $\tau_{z_1}$ acts on $V^\chi\cong E[\ell]\times E[\ell]$ by the formula
\[ \tau_{z_1}(x_0,x_1)\tau_{z_1}=\bigl(\epsilon\chi(\sigma_0)\tau_{z_1}x_0,-\epsilon\chi(\sigma_0)\tau_{z_1}x_1\bigr); \]\\

\emph{Case 2:} $\chi\not=\bar{\chi}$. Here $\tau_{z_1}$ stabilizes neither $V^\chi$ nor $V^{\bar{\chi}}$, but instead interchanges these two components. More precisely, $\tau_{z_1}$ acts on 
\[ V^\chi\times V^{\bar{\chi}}=V^\chi_0\times V^\chi_1\times V^{\bar{\chi}}_0\times V^{\bar{\chi}}_1\cong E[\ell]^4 \]
by the formula
\[ \tau_{z_1}(x_0,x_1,y_0,y_1)\tau_{z_1}=\bigl(\epsilon\bar{\chi}(\sigma_0)\tau_{z_1}y_0,-\epsilon\bar{\chi}(\sigma_0)\tau_{z_1}y_1,\epsilon\chi(\sigma_0)\tau_{z_1}x_0,-\epsilon\chi(\sigma_0)\tau_{z_1}x_1\bigr). \]
Now we define a subset $U$ of $V$ as follows:\\

\emph{Case 1:}
\begin{equation} \label{U-1-eq}
U^\chi:=\bigl\{(x_0,x_1)\mid\text{$\epsilon\chi(\sigma_0)\tau_{z_1}x_0+x_0$ and $-\epsilon\chi(\sigma_0)\tau_{z_1}x_1+x_1$ generate $E[\ell]$}\bigr\};
\end{equation}

\emph{Case 2:}
\begin{equation} \label{U-2-eq}
U^\chi\oplus U^{\bar{\chi}}:=\bigl\{(x_0,x_1,y_0,y_1)\mid\text{$\epsilon\chi(\sigma_0)\tau_{z_1}x_0+y_0$ and $-\epsilon\bar{\chi}(\sigma_0)\tau_{z_1}x_1+y_1$ generate $E[\ell]$}\bigr\}.
\end{equation}
In both cases $U$ satisfies the condition required in Proposition \ref{gen-pro}. Let now $z\in\mathcal S(U)$.
\begin{lem} \label{local-2-dim-lem}
The local cohomology classes $\mathrm{res}_ze_{\bar{\chi}}\xi(z)$ and $\mathrm{res}_ze_{\bar{\chi}}\xi(zz_1)$ form a basis of the $\F_\ell$-vector space $\bigl(\bigoplus_{z'|z}H^1(H_{z'},E)[\ell]\bigr)^{\bar{\chi}}$.
\end{lem}
\begin{proof}[Sketch of proof] Note that
\[ \dim_{\mathbb F_\ell}\Big(\bigoplus_{z'|z}H^1(H_{z'},E)[\ell]\Big)^{\bar{\chi}}=2 \]
and that the isomorphism of Corollary \ref{isom-loc-cor} sends $\text{res}_ze_{\bar{\chi}}\xi(z)$ and $\text{res}_ze_{\bar{\chi}}\xi(zz_1)$ to $\alpha_{\bar{\chi}}$ and $e_{\bar{\chi}}D_{z_1}(\alpha_{z_1})$ respectively, then keep in mind \eqref{U-1-eq} and \eqref{U-2-eq} and proceed as in \cite[Lemma 5.3]{bd}. \end{proof}
In the final part of our arguments an important role will be played by
\begin{pro} \label{global-duality-pro}
If $s\in\sel{\ell}{E/H}$ and $\gamma\in H^1(H,E)[\ell]$ then
\[ \sum_v\langle\mathrm{res}_v(s),\mathrm{res}_v(\gamma)\rangle_v=0, \]
the sum being taken over all places $v$ of $H$.
\end{pro}
\begin{proof} The statement is obtained by combining the definition of the local Poitou-Tate duality recalled in \S \ref{local-duality-subsec} with the global reciprocity law for elements in the Brauer group of $H$; for details, see \cite[Ch. I, Appendix A]{m}, especially \cite[Ch. I, Theorem A.7]{m}. \end{proof}
The following two propositions represent the crucial step towards the proof of our main results.
\begin{pro} \label{one-dim-pro}
$\dim_{\mathbb F_\ell}X^{\bar{\chi}}_z=1$ for all $z\in\mathcal S(U)$.
\end{pro}
\begin{proof} As a consequence of Proposition \ref{global-duality-pro} and the first part of Proposition \ref{kol-loc-pro}, the nonzero element $\text{res}_ze_{\bar{\chi}}\xi(z)$ belongs to the kernel of the surjection
\[ \Big(\bigoplus_{z'|z}H^1(H_{z'},E)[\ell]\Big)^{\bar{\chi}}\;\longepi\;X^{\bar{\chi}}_z \]
induced by \eqref{psi-z-eq}. On the other hand, since $\text{res}_z\alpha_{\bar{\chi}}\not=0$ and the local Poitou-Tate pairing is non-degenerate, the space $X^{\bar{\chi}}_z$ does not vanish identically on $\alpha_{\bar{\chi}}$, hence $X^{\bar{\chi}}_z\not=0$. It follows that, necessarily, $X^{\bar{\chi}}_z$ is one-dimensional over $\F_\ell$. \end{proof}
\begin{pro} \label{equality-pro}
When $z$ varies in $\mathcal S(U)$, all of the $X^{\bar{\chi}}_z$ are equal.
\end{pro}
\begin{proof} We follow \cite[Proposition 5.5]{bd} and show that $X^{\bar{\chi}}_z=X^{\bar{\chi}}_{z_1}$ for all $z\in\mathcal S(U)$. Recall the definition of the map $\Psi_z$ introduced in \eqref{psi-z-eq} and apply Proposition \ref{global-duality-pro} to the class $e_{\bar{\chi}}\xi(zz_1)$, thus obtaining:
\[ \Psi_z\bigl(\text{res}_ze_{\bar{\chi}}\xi(zz_1)\bigr)+\Psi_{z_1}\bigl(\text{res}_{z_1}e_{\bar{\chi}}\xi(zz_1)\bigr)=0 \]
in $\sel{\ell}{E/H}^{\ast,\bar{\chi}}$. But Lemma \ref{local-2-dim-lem} and Proposition \ref{one-dim-pro} imply that $\Psi_z\bigl(\text{res}_ze_{\bar{\chi}}\xi(zz_1)\bigr)$ generates $X^{\bar{\chi}}_z$. It follows that $\Psi_{z_1}\bigl(\text{res}_{z_1}e_{\bar{\chi}}\xi(zz_1)\bigr)$ is nonzero and generates $X^{\bar{\chi}}_{z_1}$, whence the claim. \end{proof}

\section{Proof of Theorem \ref{main2-thm}}
 
In this section we give a proof of Theorem \ref{main2-thm}; as explained above, this will also yield Theorem \ref{main-thm}. Let $\delta:E(H)/\ell E(H)\hookrightarrow H^1(H,E[\ell])$ be the coboundary map. Moreover, let 
\[ \sha{E/H}:=\ker\Big(H^1(H,E)\longrightarrow\prod_v H^1(H_v,E)\Big) \]
be the Shafarevich-Tate group of $E$ over $H$ (an analogous notation will be adopted, below, for an abelian variety of arbitrary dimension). Theorem \ref{main2-thm} is implied by 
\begin{thm} \label{main3-thm}
The $\F_\ell$-vector spaces $\sel{\ell}{E/H}^\chi$ and $\bigl(E(H)/\ell E(H)\bigr)^\chi$ are one-dimensional, and are generated over $\F_\ell$ by the non-zero elements $\delta(\alpha_\chi)$ and $\alpha_\chi$, respectively. Furthermore, $\sha{E/H}^\chi[\ell^\infty]=\{0\}$.
\end{thm}
\begin{proof} By Proposition \ref{gen-pro}, the $X^{\bar{\chi}}_z$ generate $\sel{\ell}{E/H}^{\ast,\bar{\chi}}$ when $z$ ranges over $\mathcal S(U)$. On the other hand, each $X^{\bar{\chi}}_z$ is one-dimensional by Proposition \ref{one-dim-pro} and all the $X^{\bar{\chi}}_z$ are equal by Proposition \ref{equality-pro}. Thus we get:
\[ \dim_{\mathbb F_\ell}\sel{\ell}{E/H}^\chi=\dim_{\mathbb F_\ell}\sel{\ell}{E/H}^{\ast,\bar{\chi}}=1, \]
and $\sel{\ell}{E/H}^\chi$ is generated over $\F_\ell$ by $\delta(\alpha_\chi)$, which is non-zero as $\alpha_\chi\not=0$ in $E(H)/\ell E(H)$ by assumption and $\delta$ is injective. In light of this, the claim on $\bigl(E(H)/\ell E(H)\bigr)^\chi$ is obvious. Finally, the short exact sequence
\[ 0\longrightarrow\bigl(E(H)/\ell E(H)\bigr)^\chi\longrightarrow\sel{\ell}{E/H}^\chi\longrightarrow\sha{E/H}^\chi[\ell]\longrightarrow0 \]
together with the first two statements gives the triviality of $\sha{E/H}^\chi[\ell^\infty]$. \end{proof}

\section{Results when $\chi=\boldsymbol 1$ and consequences} \label{trivial-sec}

It seems worthwhile to explicitly formulate our results in the special case where $\chi$ is the trivial character of $G$. These statements are well known (and nowadays could probably be defined ``classical'') for elliptic curves over $\Q$, thanks to the pioneering work of Kolyvagin (cf., for example, \cite{g2}); however, their counterparts in the function field setting have been established only recently by Brown in \cite{bro2}, where a basic Kolyvagin-type strategy is combined with the abstract language of ``Heegner sheaves'' and ``Heegner modules'' that is introduced for the purpose. We think that our line of proof, which uses exclusively elementary techniques in Galois cohomology and algebraic number theory (as in \cite{bd}), is much closer in spirit to the original arguments by Kolyvagin, and in this sense represents yet another instance of the parallel between number field and function field arithmetic. A good account on the genesis and the history of Kolyvagin-style results for elliptic curves over global function fields (and the difficulties encountered in supplying fully acceptable proofs) is given in \cite[\S 3.4]{u1}.

As in the introduction, set $\alpha_K:=\tr{H}{K}(\alpha)\in E(K)$. As a corollary of the results in this paper, we obtain
\begin{thm}[Brown] \label{main-K-thm}
If $\alpha_K$ is not a torsion point then the following are true:
\begin{itemize}
\item[1.] $E(K)$ has rank one;
\vskip 2mm
\item[2.] $\dim_{\mathbb F_\ell}\sel{\ell}{E/K}=1$ for infinitely many primes $\ell\not=p$;
\vskip 2mm
\item[3.] $\sha{E/K}[\ell^\infty]=\{0\}$ for infinitely many primes $\ell\not=p$.
\end{itemize}
\end{thm}
\begin{proof} The first point is immediate from Theorem \ref{main-thm} upon taking $\chi=\boldsymbol 1$. In fact, the condition that $\alpha_K$ has infinite order in $E(K)$ is equivalent to $\alpha_{\boldsymbol 1}=\alpha_K/|G|$ being nonzero in the complex vector space $E(K)\otimes\C\subset E(H)^{\boldsymbol 1}$. But then $\dim_{\mathbb C}E(H)^{\boldsymbol 1}=1$ by Theorem \ref{main-thm}, hence
\[ \text{rank}_{\mathbb Z}\,E(K)=\dim_{\mathbb C}\bigl(E(K)\otimes\C\bigr)=1. \]
The fact that $E(K)$ has rank one implies that 
\begin{equation} \label{dim-E(K)-eq}
\dim_{\mathbb F_\ell}\bigl(E(K)/\ell E(K)\bigr)=1
\end{equation}
for almost all primes $\ell$, and then the short exact sequence
\begin{equation} \label{descent-eq}
0\longrightarrow E(K)/\ell E(K)\longrightarrow\sel{\ell}{E/K}\longrightarrow\sha{E/K}[\ell]\longrightarrow0 
\end{equation}
ensures that $\dim_{\mathbb F_\ell}\sel{\ell}{E/K}\geq1$ for all but finitely many primes $\ell$. To show the opposite inequality for infinitely many $\ell\not=p$, let $\ell$ be any of the (infinitely many) primes for which Theorem \ref{main3-thm} holds. In particular we can assume that $E[\ell](H)=\{0\}$; hence the group $H^1\bigl(\gal{H}{K},E[\ell](H)\bigr)$ is trivial, and the inflation-restriction sequence in Galois cohomology gives an injection 
\[ \sel{\ell}{E/K}\;\longmono\;\sel{\ell}{E/H}^G. \]
But Theorem \ref{main3-thm} with $\chi=\boldsymbol 1$ implies that $\dim_{\mathbb F_\ell}\sel{\ell}{E/H}^G=1$, and the second statement is proved. Finally, part $3$ follows by combining sequence \eqref{descent-eq} with \eqref{dim-E(K)-eq} and part $2$. \end{proof}
The main reason of interest in Theorem \ref{main-K-thm} is the fact that it implies the Conjecture of Birch and Swinnerton-Dyer (BSD conjecture, for short) for $E_{/K}$. Indeed, Kato and Trihan have proved the following deep result (generalizing previous results of Tate and Milne).
\begin{thm}[Kato-Trihan] \label{kt-thm}
Let $A$ be an abelian variety over a function field $F$ of characteristic $p>0$. If $\sha{A/F}[\ell^\infty]$ is finite for some prime number $\ell$ then the Conjecture of Birch and Swinnerton-Dyer for $A$ is true.
\end{thm}
A proof of this theorem, which employs techniques in flat and crystalline cohomology, is given in \cite{kt}. 
\begin{rem}
At present, the implication in the Kato-Trihan theorem must be seen as a peculiarity of the function field setting: despite much effort in this direction, no general result of this kind is known for abelian varieties over number fields.
\end{rem}
When combined with Theorem \ref{kt-thm}, part $3$ of Theorem \ref{main-K-thm} yields the remarkable 
\begin{cor} \label{bsd-cor}
With notation as in the rest of the paper, if $\alpha_K$ is not a torsion point then the Conjecture of Birch and Swinnerton-Dyer holds for $E_{/K}$. In particular, $\sha{E/K}$ is finite.
\end{cor}
\begin{rems} \label{bsd-rems}
1. The validity of the BSD conjecture for $E_{/K}$ asserted by Corollary \ref{bsd-cor} descends to $E_{/F}$: see \cite[Theorem 1.13.1]{bro2} for precise statements.\\
2. As pointed out in \cite{bro2}, the finiteness of $\sha{E/F}$ is equivalent (by work of Artin, Tate and Milne) to the so-called Conjecture of Tate for the elliptic surface $\mathcal E_{/\mathbb F_q}$ over $\mathcal C$ which is a proper smooth model of the N\'eron model of $E_{/F}$ (cf. \cite[\S 1]{bro1}, \cite[\S 1.1]{bro2}).\\
3. A different, highly geometric approach to the BSD conjecture for elliptic curves over function fields (via non-vanishing results for twists of $L$-functions) is proposed by Ulmer in \cite{u2}. As explained in \cite[\S 3.8]{u1} (where the reader can find a brief overview of the main results of \cite{u2}), the statements in \cite[Theorem 1.2]{u2} together with a suitably general Gross-Zagier formula imply the BSD conjecture for elliptic curves over function fields of characteristic $p>3$ whose $L$-functions vanish of order $\leq1$ at the critical point. Note that the validity of the weak BSD conjecture (i.e., the equality of the algebraic and the analytic ranks) for an elliptic curve $E$ over $\Q$ satisfying this analytic condition is established by combining the results of Kolyvagin with the modularity of $E$ and the Gross-Zagier formula.\\ 
4. In the special case $F=\F_q(T)$ (i.e., when $\cC=\mathbb P^1_{/\mathbb F_q}$), R\"uck and Tipp prove in \cite{rt} a function field analogue of the Gross-Zagier formula. In particular, the Heegner point $\alpha_K$ has infinite order (i.e., the assumption of Theorem \ref{main-K-thm} is satisfied) if and only if the $L$-series of $E_{/K}$ does not vanish at $1$ (cf. \cite[Theorem 4.2.1]{rt}).\\  
5. For general Gross-Zagier formulas over global function fields the reader is referred to work in progress by Ulmer (\cite{u3}).
\end{rems}
 
\section{Final remarks} 

We conclude with a few considerations of a general nature. First of all, the arguments in this paper extend to the case of an abelian variety $A$ over $F$ which is a quotient of the jacobian of the Drinfeld modular curve $X_0(\n)$. In this situation it is expected that if $d$ is the dimension of $A$ then $\dim_{\mathbb C}A(H)^\chi=d$ provided that $\alpha_\chi\not=0$ (see \cite[\S\S 10-11]{g1} for details over number fields).

Coming back to the one-dimensional case (i.e., to an elliptic curve $E_{/F}$ as above), Corollary \ref{bsd-cor} shows that in our function field setting the knowledge of the behaviour of $\alpha_K$ gives (thanks to the result of Kato and Trihan) a much stronger control on the arithmetic of $E$ than the one guaranteed by the corresponding information for elliptic curves over $\Q$: in fact, if $\alpha_K$ has infinite order then not only do we get that $E(K)$ has rank one but we also know that the complete Birch and Swinnerton-Dyer conjecture for $E$ is true. In characteristic $0$, instead, we know that the weak BSD conjecture is true for elliptic curves over $\Q$ of analytic rank $\leq1$ (cf. Remark \ref{bsd-rems}.3 above). However, as remarked by Ulmer in \cite[\S 3.4]{u1}, over function fields it is rather difficult to explicitly compute the modular parametrization $\pi_E$ and thus the Heegner point $\alpha_K$, so the hypothesis of Corollary \ref{bsd-cor} is usually hard to check (in contrast to what happens over $\C$, where by now the computation of Heegner points has been efficiently implemented). Hence we can safely say that, at least with our present state of knowledge, in the function field case Heegner points give finer information on the arithmetic of an elliptic curve than they do over number fields, but this information is harder to access than in characteristic zero.
  
Finally, a natural question arises: can the ``Heegner hypothesis'', i.e. the assumption that all prime divisors of $\n$ split in the imaginary quadratic field $K$, be relaxed? When dealing with an elliptic curve $E$ over $\Q$ of conductor $N$ (which we suppose semistable, for simplicity) this can be done by requiring that no prime factor of $N$ ramifies in $K$ and that the number of primes dividing $N$ which are inert in $K$ is \emph{even}. This allows one to replace the modular curve $X_0(N)_{/\mathbb Q}$ by a Shimura curve $X_{/\mathbb Q}$ attached to the indefinite quaternion algebra over $\Q$ whose discriminant is the product of the primes which divide $N$ and are inert in $K$. The curve $X$ parametrizes $E$ in much the same way as $X_0(N)$ does, and the Heegner point $\alpha$ can be replaced by the image on $E$ of a suitably defined Heegner point on $X$. Then, since Heegner points on Shimura curves enjoy the same formal ``Euler system'' properties of those on modular curves, the original arguments of Bertolini and Darmon appearing in \cite{bd} carry over \emph{verbatim} to this more general situation. With this in mind, we expect that something similar can be done also in the function field setting by working with Drinfeld analogues of Shimura curves (the simplest instance of which are Drinfeld's modular curves $X_0(\n)$). Unfortunately, although analogues of this kind have already been introduced (the foundations of the subject having been laid by Laumon, Rapoport and Stuhler in \cite{lrs}), much of their arithmetic is still to be explored, and the relevant Gross-Zagier formulas, which would give an important analytic motivation to the study, remain to be proved (see \cite{u1} for remarks on these topics). We plan to return to this circle of ideas in a future work.

\end{document}